\documentclass[12pt]{article}

\usepackage{amsmath,amssymb,amsthm,mathtools}

\usepackage[a4paper,margin=1in]{geometry}

\usepackage{xcolor}

\usepackage{setspace}
\usepackage{enumitem}

\theoremstyle{plain}
\newtheorem{theorem}{Theorem}
\newtheorem{lemma}{Lemma}
\newtheorem{corollary}{Corollary}
\theoremstyle{remark}
\newtheorem{remark}{Remark}

\newcommand{\C}{\mathbb{C}}
\newcommand{\dd}{\mathrm{d}}

\usepackage{titlesec}

\usepackage{url}
\usepackage{tcolorbox}

\usepackage{hyperref}

\usepackage{fancyhdr}
\titleformat{\subsection}[block]
  {\normalfont\large\bfseries\centering}
  {\thesubsection}{1em}{}
\usepackage[blocks]{authblk}

\newcommand{\Hper}{H_{\mathrm{per}}}
\newcommand{\Gspace}{G_\sigma^{p/2}}

\title{\textbf \large  {SPATIAL AND TEMPORAL ANALYTICITY OF SOLUTIONS TO SEMI-LINEAR PARABOLIC SYSTEMS WITH ANALYTIC NONLINEARITIES}}
\author[1,2]{Nubogh B. Qumsiyeh Al-Atrash }
\affil[1]{Department of Mathematics, Birzeit University, Birzeit, Palestine}
\affil[2]{Department of Technology, Bethlehem University, Bethlehem, Palestine}
\affil[ ]{\texttt{nuboghq@bethlehem.edu}} 

\author[3,4]{Edriss S. Titi}
\affil[3]{Department of Applied Mathematics and Theoretical Physics, University of Cambridge, Cambridge CB3 0WA, UK}
\affil[ ]{\texttt{edriss.titi@maths.cam.ac.uk}}
\affil[4]{Department of  Mathematics, Texas A\&M University, College Station, TX 77843, USA}
\affil[ ]{\texttt{titi@math.tamu.edu}}

\begin{document}

\maketitle

\noindent\textbf{Abstract.}
It is shown that the solutions to a wide class of semi-linear parabolic equations with analytic (holomorphic) nonlinearity, with spatially periodic boundary conditions, are analytic in time with values in the  Gevrey class of analytic functions with respect to the spatial variable. 



 \subsection*{1. Introduction}
A characteristic property of parabolic equations is that their solutions are smoother than the initial data. It is therefore important to understand how regularity of the solutions changes in time \cite{Evans}. In this work, we focus on the regularity of strong solutions to semi-linear parabolic equations with analytic (entire) nonlinearity. Specifically, we show that for such a class of equations, subject to periodic boundary conditions, the solutions become analytic in time with values in the  Gevrey class of analytic functions with respect to the spatial variable.

An effective approach to study the analyticity in time of solutions is by extending the time variable into the complex plane. Analytical continuation in complex time provides insight into the behavior of the solutions on the real time axis and helps in identifying the mechanisms that limit their regularity \cite{Evgrafov2019, FoiasTemam1979, FoiasTemam1989}. This analytic extension exists within a specific neighborhood of the positive real time axis. The results reported in this work directly extend the time analyticity results of Foias and Temam \cite{FoiasTemam1989} for the Navier-Stokes equations to the broader class of analytic entire nonlinear parabolic equations studied by Ferrari and Titi \cite{FerrariTiti1998}. The authors of (\cite{FerrariTiti1998},\cite{FoiasTemam1989}) rely on the Gevrey functional framework to establish the existence of analytic solutions for analytic nonlinear PDEs posed on periodic domains. In this setting, spatial analyticity is characterized through the exponential decay property of the Fourier coefficients of the solutions with respect to the wave-number. 

For simplicity, we  consider the following scalar semi-linear parabolic initial value problem 
\begin{align}
u_t - \nu\Delta u + G(u,\nabla u) &= 0,\label{System}\\\quad 
    u(x,0)&=u_0,\notag
\end{align}
in  $\mathbb{R}^n$, subject spatially periodic boundary conditions with fundamental periodic domain \(\Omega=[0, L]^n\). Here, the diffusion/viscosity coefficient, $\nu$, is a positive constant, $\Delta$ is the Laplace operator, and $G$ is a real analytic (entire)  function in both $u$ and $\nabla u$. Notably, our result equally extends to system of equations of similar structure instead of a scalar equation.

Let \( A = I - \Delta \) and \( F(u,\nabla u)= G(u,\nabla u)-u\), therefore, the initial value problem \eqref{System} is equivalent to the evolution equation
\begin{align}
    u_t + \nu Au +  F(u,\nabla u) &= 0,\label{1}\\\quad  
      u(x,0) &= u_0.\notag
\end{align}

The Gevrey regularity of problem \eqref{1} was previously investigated in  \cite{FerrariTiti1998}, where it was shown that if the initial data belong to the Sobolev space $\Hper^p(\Omega)$, with $p> n/2$, then problem \eqref{1} admits a unique regular solution for sufficiently small positive $t>0$. Moreover, this solution lies, for some $t \ge 0$, in the Gevrey class $G_{t}^{p/2}(\Omega)=D(A^{p/2}e^{t A^{1/2}})$, defined below.

In the present work, we prove that for initial data in $\Hper ^p (\Omega)$, with  $p$ chosen according to the structure of the nonlinearity, there exists a unique regular solution to problem \eqref{1} that is analytic in time, in a complex neighborhood of an interval $(0,T^*)$ with values in the Gevrey class of functions (in the spatial variable) where the length of the time interval $T^*$ depends only on the  $\Hper^p$ norm of the initial data, $\nu$ and $L$. 

The proof naturally is achieved by combining ideas and tools from \cite{FerrariTiti1998}, \cite{FoiasTemam1979} and \cite{FoiasTemam1989}. More specifically, the argument relies on Galerkin approximation, which is first constructed for real time and then extended to complex time. That is, we treat the time  as a complex variable $\zeta$ and regard the approximate Galerkin solutions as complex-valued functions. This leads to a finite system of complex ordinary differential equations with an analytic entire nonlinear vector field, for which we establish the existence of analytic solutions in time in a suitable complex domain containing a neighborhood of the positive real axis. After establishing the required estimates for the complex Galerkin system we pass to the limit employing the relevant compactness theorems and the vector-valued version of Vitali’s Theorem (Theorem 2.1 \cite{arendt_nikolski_2000}) to accomplish the analyticity of the solutions in time with values in the Gevery class of analytic functions. 

It is worth mentioning that the tools and results reported in \cite{FerrariTiti1998} have been extended to the case when system \eqref{System} and the Navier-Stokes equations are considered on the two-dimensional sphere $\mathbb{S}^2\subset\mathbb{R}^3$ instead of periodic boundary conditions (cf. \cite{Cao-Rammaha1} and \cite{Cao-Rammaha2}). Therefore, the same tools and approach employed here can be combined with \cite{Cao-Rammaha1} and \cite{Cao-Rammaha2} to achieve similar analyticity results to system \eqref{System} when the latter is  considered on the two-dimensional sphere $\mathbb{S}^2$ instead of periodic boundary conditions. Furthermore, similar analyticity results for semi-linear holomorphic parabolic systems in the whole space $\mathbb{R}^n$ have  been established in \cite{Takac}, using a completely different, yet more cumbersome, approach. 

\vspace{0.5cm}
   
\noindent \textbf{Notations}

Next, let us introduce the following definitions and notation 
\begin{itemize}
\item  The operator $A=I-\Delta $ is an unbounded, self-adjoint, positive definite operator in $L_\text{per}^2(\Omega)$, with domain $D(A)=H_\text{per}^2(\Omega)$.
    
\item More generally, the domain of $A^{p/2}$, for $p\ge0$, is denoted by $D(A^{p/2})$,  and is characterized by 
\[D(A^{p/2})=\{u\in L_\text{per}^2(\Omega):\sum_{j\in \mathbb{Z}^{n}}|u_{j}|^{2}(1+|j|^{2})^{p}<\infty\},\]
    where $u_{j}$ are the Fourier coefficients of $u$ with respect to the periodic orthonormal Fourier basis of trigonometric functions.
    
\item  $\Hper^p(\Omega)$ denotes the  Sobolev space of periodic functions with the norm 
    \[|u|_{\Hper^p(\Omega)}=\left(\sum_{j\in \mathbb{Z}^{n}}|u_{j}|^{2}(1+|j|^{2})^{p}\right)^{1/2}.\]
    Observe that $\Hper^p(\Omega)=D(A^{p/2})$ with equivalent norms. Indeed, for any $u\in \Hper^p(\Omega)$:
    \[|u|_{\Hper^p(\Omega)}=|A^{p/2}u|_{L_\text{per}^2(\Omega)}.\]
    
    \item The Gevrey class of real analytic functions to be considered in this work is 
      \[G_{\sigma}^{p/2}(\Omega):=D(A^{p/2}e^{\sigma A^{1/2}})=\{u\in L_\text{per}^2(\Omega):\sum_{j\in \mathbb{Z}^{n}}|u_{j}|^{2}(1+|j|^{2})^{p}e^{2\sigma|j|}<\infty\},\] 
     for some $\sigma > 0$, which  determines the radius of analyticity.
      
Observe that these  are Hilbert spaces with respect to the inner products 
\[
(v,w)_{G_{\sigma}^{p/2}(\Omega)}=\sum_{j\in \mathbb{Z}^{n}}v_{j}\cdot\overline{w}_{j}(1+|j|^{2})^{p}e^{2\sigma(1+|j|^{2})^{1/2}}
,\]
with the corresponding norms 
\begin{align*}
    |v|_{G_{\sigma}^{p/2}(\Omega)}=\left(\sum_{j\in \mathbb{Z}^{n}}|v_{j}|^{2}(1+|j|^{2})^{p}e^{2\sigma(1+|j|^{2})^{1/2}}\right)^{1/2}, 
\end{align*}
where the $v_{j}$ and $w_{j}$ are the Fourier coefficients of $v$ and $w$ respectively. 

Note that \[|v|_{G_{\sigma}^{\frac{p+1}{2}}(\Omega)}=\left(\sum_{j\in \mathbb{Z}^{n}}|v_{j}|^{2}(1+|j|^{2})^{p+1}e^{2\sigma(1+|j|^{2})^{1/2}}\right)^{1/2}.\]
For any wave vector $j \in \mathbb{Z}^n$ and for any  $p\ge 0$ one has
$$\sum_{j\in \mathbb{Z}^{n}}\vert{}v_{j}\vert{}^{2}(1+\vert{}j\vert{}^{2})^{p}e^{2\sigma(1+\vert{}j\vert{}^{2})^{1/2}} \leq \sum_{j\in \mathbb{Z}^{n}}\vert{}v_{j}\vert{}^{2}(1+\vert{}j\vert{}^{2})^{p+1}e^{2\sigma(1+\vert{}j\vert{}^{2})^{1/2}},
$$
which yields
\begin{align}
    \vert{}v\vert{}_{G_\sigma^{p/2}} \leq \vert{}v\vert{}_{G_\sigma^{\frac{p+1}{2}}},\label {*}
\end{align}
which is a Poincar\'e-like inequality.

\item Constants are denoted  by $C$ or $C_k, k\in \mathbb{N}_0$ and may change from line to line.

\item To obtain analyticity in time with values in the Gevrey class of analytic functions (with respect to the spatial variable), let  $\Hper ^p (\Omega,\mathbb{C})$ be the complexification of the real Hilbert space $\Hper ^p (\Omega)$ and consider problem \eqref{1} for a complex time variable $\zeta \in \mathbb{C}$, and a complex valued function $u$, such that for
 $\zeta = s e^{i\theta}$, with modulus $s = |\zeta| > 0$ and argument $|\theta| = |\operatorname{Arg} (\zeta)| < \frac{\pi}{2}$, we define the dynamic Gevrey parameter by:
\[
\sigma = \sigma(s,\theta) := s\cos\theta = \mathrm{Re}\,\zeta >0.
\]
We introduce the open complex time domain $\mathcal{D}_T$ defined by:
\[
\mathcal{D}_T := \left\{\zeta = s e^{i\theta} \in \mathbb{C} :\ 0 < s < T(\theta), \ \ |\theta| < \frac{\pi}{2}\right\},
\]
where $T: (-\pi/2, \pi/2) \to (0,\infty)$ is a continuous, positive function that will be specified explicitly in the details of the proof.

\item We denote by $P_N$ the $L_\text{per}^2$ orthogonal projection onto the finite-dimensional subspace:
\[
H_N = \operatorname{span} \left\{ e^{2\pi i j \cdot x /L} ,  |j| \leq N \right\}.
\]
\item By linearity, the operator $A$ (respectively, $P_N $) extends to a self-adjoint operator in $\Hper^p(\Omega,\mathbb{C})$ (respectively, to the orthogonal projection $P_N $ in $\Hper ^p (\Omega,\mathbb{C})$  (see \cite{BuehlerSalamon2017}, \cite{moslehian2022similarities} and\cite{Sabourova2007}). The same notation for the extension of the linear operator $A$ (respectively, $P_N$) in the complexified space will be retained.

\vspace{0.5cm}
\noindent \textbf{Galerkin Approximation}

The sequence $u^N(x, t) = \sum_{ |j| \leq N} \alpha_j^N(t)  e^{2\pi i j \cdot x / L}\in H_N$  of approximating solutions to the problem \eqref{1} is defined as the solutions  of the finite-dimensional Galerkin system
\begin{align*}
    \frac{du^N}{dt} + \nu A u^N + P_N F(u^N,\nabla u^N) &= 0,\\
    u^N(0)= P_N u_0 .
\end{align*}
\end{itemize}

Finally, we recall Lemma 1 and Lemma 2 below, from \cite{FerrariTiti1998}, where the results are extended to the complex-valued setting. This extension is possible because the proofs of Lemmas rely on the structure of the Gevrey norms (which depend only on the moduli of the Fourier coefficients and are unchanged in the complex case), convolution estimates, the analyticity of $F$, the majorant method, and the absolute convergence of series  — all of which remain valid over the complex plane $\mathbb{C}$ (see,  e.g., \cite{moslehian2022similarities} and \cite{Sabourova2007}).

\begin{lemma} 
If $u$ and $v$ are in the class $G_{\sigma}^{p/2}(\Omega)$ with $p>n/2,$ then $uv\in G_{\sigma}^{p/2}(\Omega)$ and there exists a constant $C_{p}$, independent of $\sigma$, such that
\[
|uv|_{G_{\sigma}^{p/2}(\Omega)}\le C_{p}|u|_{G_{\sigma}^{p/2}(\Omega)}|v|_{G_{\sigma}^{p/2}(\Omega)},
\]
i.e. the Hilbert space $G_{\sigma}^{p/2}(\Omega)$ is a Banach algebra.
\end{lemma}

Assume, for now, that $F$, in equation \eqref{1}, does not depend on $\nabla u$, i.e.,  \(F:\mathbb{C} \to \mathbb{C}\) is a holomorphic entire function in $u$. Hence $F$ admits a convergent power series expansion for every $u\in\mathbb{C}$ such that:
\[
F(u) = \sum_{j=0}^{\infty} a_j u^j,a_j\in \mathbb{C},\]
and it has a majorant analytic function given by 
\begin{align}
    \quad g(s) = \sum_{j=0}^{\infty} |a_j| s^j,\label{2}
\end{align} which converges for all \( s\in \mathbb{R} \) (\cite{{hille1976ordinary}}). 

\begin{remark}
 As we stated above, for simplicity we are dealing here with a scalar case,  and for now we assume that $F$ in  \eqref{1} does not depend on $\nabla  u$. If we are dealing with a system of equations instead of the scalar case, then $F$ becomes a vector valued analytic (entire) function of several variables (the unknown vector components), so its domain is $\mathbb{C}^m$   for some integer $m$, and the proof follows essentially the scalar case with the necessary notational adjustments. 
\end{remark}

\begin{lemma}
    Let $u\in G_{\sigma}^{p/2}(\Omega)$ and $F(u)$ be an analytic (entire) function with a majorising function $g$ as described in \eqref{2}. Then $F(u)\in G_{\sigma}^{p/2}(\Omega)$ and
\[
|F(u)|_{G_{\sigma}^{p/2}(\Omega)}\le(1+C_{p}^{-1})g(C_{p}|u|_{G_{\sigma}^{p/2}(\Omega)}).
\]
\end{lemma}

 We will discuss the following two cases:
 \subsection*{ 2. Nonlinearity of the form \texorpdfstring{$F(u)$}{F(u)} }
First, consider the scalar complexified equation:
\begin{align}
    \frac{du}{d\zeta} + \nu Au + F(u) &= 0 ,\quad\label{3}\\
    u(0) = u_0.  \label{4}
\end{align}

The following theorem is the first main result of this paper.

\begin{theorem}
Let $u_0 \in  \Hper ^p (\Omega)$ with $p > \frac{n}{2}$ and $|u_0|_{\Hper ^p (\Omega)} \le M_0$, for some $M_0 > 0$. There exists a continuous, positive function $T^*(\theta)$ defined for $|\theta| < \frac{\pi}{2}$ that depends only on the parameters $M_0$, $\nu$, and $L$, such that equation \eqref{3} with initial condition \eqref{4} has a unique regular solution $u$ inside the domain:
\[
\mathcal{D}_{T^*} := \left\{\zeta = s e^{i\theta} \in \mathbb{C} :\ 0 < s < T^*(\theta), \ \ |\theta| < \frac{\pi}{2}\right\},
\]
and the map
\[
\zeta \longmapsto e^{\sigma A^{1/2}} A^{p/2} u(\zeta), \quad \text{with } \sigma = \sigma(s,\theta) = s\cos\theta,
\]
is analytic in $\mathcal{D}_{T^*}$.
\end{theorem}

\begin{proof} 
To establish the existence of a solution that is analytic in time, we proceed through the following steps. 
\begin{enumerate}
    
\item Consider the complexified form of the Galerkin approximation of equations \eqref{3} and \eqref{4}, namely, the finite-dimensional complex ordinary differential  system in $P_N L_\text{per}^2(\Omega,\mathbb{C})$  
\begin{align}
\frac{du^N}{d \zeta} + \nu A u^N + P_N F(u^N) &= 0, \label{5}\\
u^N( 0) &= P_N u_0,\label{6}
\end{align}
where $\zeta \in \mathbb{C}$ and $u^N$ maps $\mathbb{C}$ or an open set of $\mathbb{C}$ into $P_N L_\text{per}^2(\Omega,\mathbb{C})$.

The sequence $
u^N(x, \zeta) = \sum_{ |j| \leq N} \alpha_j^N(\zeta)  e^{2\pi i j \cdot x / L}$  of approximating solutions to \eqref{3}-\eqref{4} is defined as solutions of the complexified Galerkin system \eqref{5} and \eqref{6}. Since the right-hand side of \eqref{5}, expressed as:
\begin{align*}
    \frac{du^N}{d \zeta} =- \nu A u^N - P_N F(u^N) ,
\end{align*}  
is holomorphic in the coefficients $\alpha_j^N(\zeta)$ near the initial data $u^N(0)=P_Nu_0$, a direct application of the  Cauchy-Kowalevskia theorem (see \cite{FollandPDE}, \cite{Tsogtgerel2014} and   \cite{hille1976ordinary} ) guarantees that for each $N$, there exists a unique analytic solution $u^N(\zeta)$ defined in a complex neighborhood of the origin.

 \item Derive \textit{a prior} estimate for \( u^N=u^N(\zeta) .\) Apply \( A^{p/2} e^{\sigma A^{1/2}} \) to \eqref{5} and take the  $L_\text{per}^2 (\Omega)$ \text{ inner product with } \( A^{p/2} e^{\sigma A^{1/2}} \) $u^N$,  \text{we obtain} 

\[
\left( A^{p/2} e^{\sigma A^{1/2}} \frac{\dd u^N}{\dd\zeta} , A^{p/2} e^{\sigma A^{1/2}} u^N\right)_ {L^2(\Omega)}
\]
\[
+ \left( A^{p/2} e^{\sigma A^{1/2}}\nu  Au^N, A^{p/2} e^{\sigma A^{1/2}} u^N \right)_ {L^2(\Omega)}
\]
\[
+ \left( A^{p/2} e^{\sigma A^{1/2}} P_N F(u^N), A^{p/2} e^{\sigma A^{1/2}} u^N \right)_ {L^2(\Omega)}= 0.
\]
Multiply the above by $e^{i\theta}$ for fixed \(\theta\), \(|\theta| < \frac{\pi}{2}\), and take its real part, to conclude

\begin{align}
& \operatorname{Re}\Big( e^{i\theta}
\big( A^{p/2} e^{\sigma A^{1/2}} \frac{\dd u^N}{\dd\zeta},
A^{p/2} e^{\sigma A^{1/2}} u^N \big)_{L^2(\Omega)} \Big) \notag \\
& + \operatorname{Re}\Big( e^{i\theta}
\big( A^{p/2} e^{\sigma A^{1/2}} \nu A u^N,
A^{p/2} e^{\sigma A^{1/2}} u^N \big)_{L^2(\Omega)} \Big) \notag \\
& + \operatorname{Re}\Big ( e^{i\theta}
\big( A^{p/2} e^{\sigma A^{1/2}} P_N F(u^N),
A^{p/2} e^{\sigma A^{1/2}} u^N \big)_{L^2(\Omega)} \Big)
= 0.\label{7}
\end{align}
Take  \(\zeta=s e^{i\theta}, \sigma=\sigma(s,\theta):=s\cos\theta\) and define
\[v(s) = e^{\sigma A^{1/2}} u^N(s e^{i\theta})=e^{s\cos\theta A^{1/2}} u^N(s e^{i\theta}),
\]
so, \[
\frac{\dd v}{\dd s} = \cos\theta A^{1/2} e^{\sigma A^{1/2}} u^N(s e^{i\theta}) + e^{\sigma A^{1/2}} \frac{\dd u^N(s e^{i\theta})}{\dd s},
\]
and 
\begin{align}
    e^{\sigma A^{1/2}} \frac{\dd u^N(s e^{i\theta})}{\dd s} = \frac{\dd v}{\dd s} -  \cos\theta A^{1/2} v(s).\label{8}
\end{align}

Substituting \(\dfrac{\mathrm{d} u^N}{\mathrm{d} \zeta} = e^{-i\theta} \dfrac{\mathrm{d} u^N}{\mathrm{d} s},\) using \eqref{8} and Young's inequality, the derivative term of equation \eqref{7} yields: 

\begin{align}
   & \operatorname{Re}\left(
  e^{i\theta}\,
  (
    A^{p/2} e^{\sigma A^{1/2}} \frac{\dd u^N}{\dd\zeta},\;
    A^{p/2} e^{\sigma A^{1/2}} u^N)
  _{L^2(\Omega}\right)\notag \\
&\quad= \operatorname{Re}\left(\ (A^{p/2} e^{\sigma A^{1/2}} \frac{\dd u^N}{\dd s} , A^{p/2} e^{\sigma A^{1/2}} u^N)_ {L^2(\Omega)} \right)\notag \\
&\quad= \operatorname{Re}\left((A^{p/2} (\frac{\dd v}{\dd s} - \cos\theta A^{1/2} v), A^{p/2}  v) _ {L^2(\Omega)}\right)\notag \\
&\quad= \operatorname{Re}\left((A^{p/2} \frac{\dd v}{\dd s}, A^{p/2}  v) _ {L^2(\Omega)}- ( \cos\theta A^{\frac{p+1}{2}} v,A^{p/2}  v) _ {L^2(\Omega)}\right)
\notag \\
&\quad= \frac{1}{2} \frac{\dd}{\dd s} |A^{p/2} v|^2_{L^2}- \operatorname{Re}\left(\cos \theta (A^{\frac{p+1}{2}} v,A^{p/2}  v) _ {L^2(\Omega)}\right)\notag\\
&\quad \geq \frac{1}{2} \frac{\dd}{\dd s} |A^{p/2} v|^2_{L^2}-  \left| A^{p/2} v \right|_{L^2} \left| A^{\frac{p+1}{2}} v \right|_{L^2}
\notag\\
&\quad \geq
 \frac{1}{2} \frac{\dd}{\dd s} |u^N|^2_{G_\sigma^{p/2}}-\epsilon
    |u^N|^2_{G_\sigma^{\frac{p+1}{2}}}-C_\epsilon | u^N|^2_{G_\sigma^{p/2}},\label{9}
\end{align} 
\vspace{0.5cm}
where we used Young's inequality above for some $\epsilon >0$ to be chosen later.
 
On the other hand, the second term of equation \eqref{7} is
\begin{align}
&\operatorname{Re}\left(e^{i\theta}( A^{p/2} e^{\sigma A^{1/2}} \nu Au^N, A^{p/2} e^{\sigma A^{1/2}} u^N )_ {L^2(\Omega)}\right) \notag \\
&\quad = \operatorname{Re}\left( e^{i\theta}(\nu  A^{\frac{p+1}{2}} e^{\sigma A^{1/2}} u^N, A^{\frac{p+1}{2}} e^{\sigma A^{1/2}} u^N )_ {L^2(\Omega)}\right) \notag \\
&\quad = \nu \cos\theta|u^N|_{G_\sigma^{\frac{p+1}{2}}}^2.\label{10}
\end{align}
 From \eqref{7}, \eqref{9} and \eqref{10} we get 
\begin{align}
& \frac{1}{2} \frac{\dd}{\dd s} |u^N|_{G_\sigma^{p/2}}^2 + (\nu \cos\theta -\epsilon)|u^N|_{G_\sigma^{\frac{p+1}{2}}}^2 -C_{\epsilon}|u^N|_{
G_\sigma^{p/2}
}^2 \notag \\
&\leq -\operatorname{Re}\left( e^{i\theta}( A^{p/2} e^{\sigma A^{1/2}} P_N F(u^N), A^{p/2} e^{\sigma A^{1/2}} u^N)_ {L^2(\Omega)}\right).\label{11}
\end{align}
 Applying Lemma 1 and Lemma 2 and using the fact that  $|P_N F(u^N)|_{G_\sigma^{p/2}}\leq|F(u^N)|_{G_\sigma^{p/2}} $, the absolute value of the right-hand side of \eqref{11} satisfies
\begin{align*}
& \Big|-\operatorname{Re}\left( e^{i\theta}( A^{p/2} e^{\sigma A^{1/2} } P_N F(u^N), A^{p/2} e^{\sigma A^{1/2}} u^N)_{L^2(\Omega)}\right)\Big| \notag \\
&\quad \leq |P_N F(u^N)|_{G_\sigma^{p/2}}|u^N|_{G_\sigma^{p/2}} \notag \\
&\quad \leq |F(u^N)|_{G_\sigma^{p/2}}|u^N|_{G_\sigma^{p/2}} \notag \\
&\quad \leq  (1+C_p^{-1})g(C_p |u^N|_{G_\sigma^{p/2}})|u^N|_{G_\sigma^{p/2}}.
\end{align*}
Therefore, we have 
\begin{align*}
\frac{1}{2} \frac{\dd}{\dd s} |u^N|_{G_\sigma^{p/2}}^2 + (\nu \cos\theta-\epsilon) |u^N|_{G_\sigma^{\frac{p+1}{2}}}^2 & \leq (1+C_p^{-1})g(C_p |u^N|_{G_\sigma^{p/2}})|u^N|_{G_\sigma^{p/2}}\notag \\&\quad+C_{\epsilon}|u^N|_{G_\sigma^{p/2}}^2. 
\end{align*}
Let $\epsilon = \dfrac{\nu \cos\theta}{4} > 0$, $C_\epsilon = \dfrac{1}{\nu \cos\theta}$  and using \eqref{*} we get
\begin{align*}
  \frac{1}{2} \frac{\dd}{\dd s} |u^N|_{G_\sigma^{p/2}}^2+ \frac{3\nu \cos\theta}{4} |u^N|_{G_\sigma^{p/2}}|u^N|_{G_\sigma^{\frac{p+1}{2}}} & \leq (1+C_p^{-1})g(C_p |u^N|_{G_\sigma^{p/2}})|u^N|_{G_\sigma^{p/2}}\notag \\&\quad+\frac{1}{\nu \cos\theta}|u^N|_{G_\sigma^{p/2}}^2. 
\end{align*}
Observe that for any $\delta>0$ the above gives
\begin{align*}
  \frac{1}{2} \frac{\dd}{\dd s} (|u^N|_{G_\sigma^{p/2}}^2+\delta)+ \frac{3\nu \cos\theta}{4} |u^N|_{G_\sigma^{p/2}}|u^N|_{G_\sigma^{\frac{p+1}{2}}} & \leq (1+C_p^{-1})g(C_p |u^N|_{G_\sigma^{p/2}})|u^N|_{G_\sigma^{p/2}}\notag \\&\quad+\frac{1}{\nu \cos\theta}|u^N|_{G_\sigma^{p/2}}^2. 
\end{align*}
Let $y(s)=|e^{\sigma A^{1/2}} u^N(s e^{i\theta})|=|u^N|_{\Gspace}$ and $y_\delta(s)=\sqrt{y^2(s)+\delta}$, since $g$ is monotonic non-decreasing then the above inequality implies 
\begin{align*}
  \frac{1}{2} \frac{\dd}{\dd s} y^2_\delta (s)+ \frac{3\nu \cos\theta}{4} |u^N|_{G_\sigma^{p/2}}|u^N|_{G_\sigma^{\frac{p+1}{2}}} & \leq (1+C_p^{-1})g(C_p y_\delta )y_\delta \notag \\&\quad+\frac{1}{\nu \cos\theta}y_\delta^2. 
\end{align*}
We drop the positive term $\frac{3\nu \cos\theta}{4} |u^N|_{G_\sigma^{p/2}}|u^N|_{G_\sigma^{\frac{p+1}{2}}}$ and divide by $y_\delta$ to obtain
$$
\frac{\dd}{\dd s} y_\delta (s)  \leq (1+C_p^{-1})g(C_p y_\delta ) +\frac{1}{\nu \cos\theta}y_\delta. 
$$
Integrating from $0$ to $s$ yields
$$
    y_\delta(s) \leq \int_0^s  \!\bigl(1 + C_p^{-1}\bigr)\,
    g \bigl(C_p  y_\delta(t)\bigl) dt+ \int_0^s  \frac{1}{\nu \cos\theta}  y_\delta(t)dt
    + y_\delta(0),
$$
and by letting $\delta \to 0$ we conclude

\begin{align}
    y(s) &\leq \int_0^s  \!\bigl(1 + C_p^{-1}\bigr)\,
    g \bigl(C_p y(t)\bigl) dt+ \int_0^s  \frac{1}{\nu \cos\theta} y(t)dt
    + y(0).\label{12}
\end{align}
By the continuity of $y(s)$ with respect to \(s\), and since
$|u_0^N|_{G_{0}^{p/2}} \leq C_{0} |u_0|_{H^p}$,  there exists a $T^*_N(\theta)>0,$ such that
$y(s)=|u^N|_{G_{\sigma}^{p/2}} \leq  2|u_0^N|_{G_{0}^{p/2}} + 1 \leq 2C_{0} |u_0|_{H^p} + 1$ for \(s \in [0, T^*_N(\theta)]\).
Since $g(s)$ is a nondecreasing function for $s \geq 0$,  from \eqref{12} we have for all \(s \in [0, T^*_N(\theta)]\)
\begin{align}
    |u^N(se^{i\theta})|_{G_{\sigma}^{p/2}} \leq s  &\left ((1 + C_p^{-1}) g \left( C_p( 2C_{0} |u_0|_{H^p} + 1 \right) )+   \frac{1}{\nu \cos\theta}(2C_{0} |u_0|_{H^p} + 1)  \right)\notag \\ &  + C_0|u_0|_{H^p}.\label{13}
\end{align}
It remains to verify that \( T_N^*(\theta) \) does not tend to 0 as \( N \to \infty \). Observe that for any fixed $\theta$ satisfying $|\theta|<\pi/2$, if we require the right-hand side of \eqref{13} to be less than or equal to 
\( 2C_0|u_0|_{H^p} + 1 \), then inequality \eqref{13} will be satisfied for every  \( s \in [0, T^*(\theta)] \), where
\begin{align}
    T^*(\theta) = \frac{C_0|u_0|_{H^p} + 1}{(1 + C_p^{-1}) g(C_p(2C_0|u_0|_{H^p} + 1)) +  \frac{1}{\nu \cos\theta}(2C_{0} |u_0|_{H^p} + 1)}.\label{14}
\end{align}
Consequently, we establish a lower bound for the maximal existence time of solutions to the Galerkin approximation system, yielding $0 < T^*(\theta) \leq T_N^*(\theta)$ for all $N$. Moreover, we also have the following the uniform estimate
\begin{align}
   \sup_{0 \leq s\leq T^*(\theta)} \left| u^N(se^{i\theta}) \right|_{G_{\sigma}^{p/2}} \leq\ 2C_{0} |u_0|_{H^p} + 1 = C, \label{15} 
\end{align}
holds for all $N$. Here, $C$ is independent of both $N$ and $\theta$, whereas the domain of validity $[0, T^*(\theta)]$ depends explicitly on $\theta$.

This shows that the solution \( u^N \) of \eqref{5}, which is defined and analytic in a neighborhood of $\zeta=0$, actually extends to an analytic solution of this equation in an open set of $\C$ containing the domain
\begin{align}
   \mathcal D_{T^*}:=\{\zeta=se^{i\theta}\in\mathbb C:\ 0<s<T^*(\theta),\ |\theta|<\tfrac{\pi}{2}\}.\label{16}
\end{align} 
That is 
\begin{align}
     \sup_{\zeta\in \mathcal D_{T^*}} \left| u^N(\zeta) \right|_{\Gspace} \leq  C. \label{17} 
\end{align}

We now pass to the limit $N \to\infty$.  

We will use the vector-valued version of Vitali’s Theorem (Theorem 2.1 \cite{arendt_nikolski_2000}).

\textbf{Vitali’s Theorem (Theorem 2.1 \cite{arendt_nikolski_2000})}
Let $U\subset\mathbb{C}$ be an open connected set. Let $(f_i)_{i \in I}$ be a net of holomorphic functions on $U$ with values in $X$ that is locally bounded (i.e., for all $z \in U$ there exists a neighborhood on which $(f_i)_{i \in I}$ is bounded).
Then the following assertions are equivalent:
\begin{enumerate}
    \item The net $(f_i)_{i \in I}$ converges uniformly on all compact subsets of $U$ to a holomorphic function $f : U \to X$;
    \item The set 
    \(
    U_0 := \left\{ z \in U : \lim_{i} f_i(z) \text{ exists} \right\}
    \),
    has an accumulation point in $U$;
    \item There exists $z_0 \in U$ such that $\lim_{i} f_i^{(k)}(z_0)$ exists for all $k \in \mathbb{N}$.
\end{enumerate}

Here set $U=\mathcal{D}_{T^*}$ (which is an open and connected in $\mathbb{C}$), $f_N(\zeta)=u^N(\zeta)$, and $X=H^{p}(\Omega;\mathbb{C})$. By the continuous embedding $G_{\sigma}^{p/2} \hookrightarrow H^{p}(\Omega;\mathbb{C}) $ and estimate \eqref{17}, the sequence of holomorphic functions $u^N: \mathcal{D}_{T^*} \to H^{p}(\Omega;\mathbb{C}) $ is locally bounded in $H^{p}(\Omega;\mathbb{C})$. 

Define the real axis slice $U_0:=(0, T^*(0))\subset \mathcal{D}_{T^*}$. for the real time $\zeta=s\in \mathbb{R}$, $u^N(s)$ is the real-time Galerkin sequence, which already known to converge to the unique strong solution $u(s)$ in $H^{p}(\Omega;\mathbb{C})$ \cite{FerrariTiti1998}. Thus $\lim_{N \to \infty}
 u^N(\zeta) $ exists for every $\zeta\in U_0$. 
 
The open interval $(0, T^*(0))$ contains accumulation points in $\mathcal{D}_{T^*}$; condition (b) of Vitali’s Theorem is satisfied. Consequently, assertion (a) holds: $u^N(\zeta)$ converges uniformly on compact subsets of $\mathcal{D}_{T^*}$ to a holomorphic limit function $u: \mathcal{D}_{T^*}\to H^{p}(\Omega;\mathbb{C})$.
 
 Finally, since $u^N(\zeta) \to u(\zeta)$ strongly in $H^p(\Omega;\mathbb{C})$ for each $\zeta \in \mathcal{D}_{T^*}$, the lower semicontinuity of the Gevrey norm Proposition A.19 \cite{RobinsonRodrigoSadowski2016}) yields
\begin{align}
\lvert u(\zeta) \rvert_{G_{\sigma}^{p/2}}
\le \liminf_{N \to \infty}
\lvert u^N(\zeta) \rvert_{G_{\sigma}^{p/2}}
\le C,\quad \forall \zeta \in \mathcal{D}_{T^*}.\label{18}
 \end{align}
Thus, the limit function $u(\zeta)$ satisfies the uniform Gevrey bound throughout $\mathcal{D}_{T^*}$, establishing that 
\(
\zeta \longmapsto e^{\sigma A^{1/2}} A^{p/2} u(\zeta)\), (with  $\sigma = \sigma(s,\theta) = s\cos\theta$)
is analytic in $\mathcal{D}_{T^*}$.

In summary, the argument carried out at the initial time $t=0$ can be repeated at
any time $t_0>0$ for which the solution exists and satisfies
\[
u(t_0) \in H^p_{\mathrm{per}}(\Omega), \qquad
|u(t_0)|_{\Hper^p} \leq M_0.
\]
That is, by restarting the Galerkin construction at time $t_0$ and using $u(t_0)$ as new
initial data, it follows that 
\begin{align*}
     \sup_{0\leq s\leq T^*} \left| u(se^{i\theta}+t_0) \right|_{\Gspace} \leq\text{Constant} = C,  
\end{align*}  
such that $\sigma=t_0+s\cos\theta$. This shows that \( u \) is a \( \Gspace \)-valued analytic function, defined at least in the following region of the complex plane:
\begin{align*}
    \bigcup_{t_0} \left\{ t_0 + \mathcal D_{T^*} \right\},
\end{align*}
where the union is taken over all \( t_0 \in (0,\infty) \) .
Hence, we obtain analyticity of the solution in a complex neighborhood of
$t_0$, with $T^* (\theta)$. The proof is complete.
\end{enumerate}
\end{proof}

\begin{corollary}
Since $u$ is analytic, we can apply Cauchy's formula to \eqref{18} to obtain a priori bounds for its derivatives ( with respect to $\zeta$) on compact subsets of $\mathcal D_{T^*}$ \cite{Evgrafov2019}. Indeed, for $\zeta\in \mathcal D_{T^*}$ and $k\in\mathbb{N},k\geq1 $, let $r>0$  such that
\[
\overline{B_r(\zeta)} \subset \mathcal D_{T^*},
\] 
we have
\[
\frac{d^k u}{d\zeta^k}(\zeta)
=
\frac{k!}{2\pi i}
\int_{|z-\zeta|=r}
\frac{u(z)}{(z-\zeta)^{k+1}}\, dz.
\]
Therefore
\[
\left\lvert \frac{d^k u}{d\zeta^k}(\zeta) \right\rvert_{\Gspace}
\le\frac{k!}{r^k}\sup_{z\in \mathcal D_{T^*}}|u(z)|_{G_{\sigma}^{p/2}}\le
\frac{k!}{r^k}\, C.
\]
In particular, if $K \subset \mathcal D_{T^*}$ is compact, then since \(\mathrm{dist}(K,\partial \mathcal D_{T^*})>0,\) there exists $r_K>0$ such that
$\overline{B_{r_K}(\zeta)} \subset \mathcal D_{T^*}$ for all $\zeta \in K$, and
\begin{align}
    \sup_{\zeta\in K}\lvert \frac{d^k u}{d\zeta^k}(\zeta)|_{G_{\sigma}^{p/2}}
\le
\frac{k!}{r_K^k}\, C.\label{19}
\end{align}
\end{corollary} 

\subsection*{3. Nonlinearity of the form  \( F(u,\nabla u) \) }
Consider the scalar equation:
\begin{align}
    \frac{du}{d\zeta} + \nu Au + F(u,\nabla u) &= 0 .\label{20}\\
    u(0) &= u_0,  \label{21}
\end{align}
Assume that \(F : \mathbb{C}\times \mathbb{C}^{n} \to \mathbb{C}
\)  is a holomorphic entire function in all its arguments; hence $F$ admits a convergent power series expansion for every $(u,\nabla u)\in\mathbb{C}\times \mathbb{C}^{n} $ such that
\[
F(u,u_{x_1},u_{x_2},...,u_{x_n})=\sum_\beta a_\beta u^{\beta_ 1}u^{\beta_ 2}_{x_1}...u^{\beta_{n+1}}_{x_n}, a_\beta \in \mathbb{C},
\]
where $\beta=(\beta_1, \beta_2, ..., \beta_{n+1})$, and $\beta_j$ are nonnegative integers.

Moreover, assume that $F$ has a majorising function $g$  given by
\begin{align}
    g(r,\rho_1,...,\rho_n)=\sum_\beta |a_\beta| r^{\beta_ 1}\rho_1^{\beta_ 2}...\rho_n^{\beta_{n+1}}, \label{22}
\end{align}
which converges for all $(r,\rho)$ $\in \mathbb{R}^{n+1}$.

Since the nonlinearity $F(u,\nabla u)$ involves spatial gradients which  affects Gevrey norms, we will distinguish between two cases according to the regularity of the initial data $u_0$.
\subsubsection*{2.1 The case $ \frac{n}{2}< p \leq \frac{n}{2}+1.$}
Assume that for every $r \in \mathbb{R}$, there exists a positive, increasing function $f(r)$ in $r$ for $r>0$, such that
\begin{align}
|g(r,\rho)|\;\le\; f(r)\bigl(1+|\rho|^{\gamma}\bigr),
\qquad 0 \le \gamma < 2, \quad \rho \in \mathbb{R}^n .\label{23}
\end{align}
This assumption on the majorising function $g$ ensures a suitable growth order of the function $F(u,\nabla u)$ with respect to the gradient variable. Since $F(u,\nabla u)$ is assumed to be an entire holomorphic function, this growth restriction significantly limits its dependence on the gradient. More precisely, by Liouville's Theorem this condition implies that $F(u,\nabla u)$ is a  linear polynomial in $\nabla u$. Consequently, we consider a particular case when
\[
F(u,\nabla u) = a(u)\cdot \nabla u + b(u),
\]
where $a(u)$ and $b(u)$ are analytic functions of $u$. In particular, under this structural assumption, it is sufficient to require $p > n/2$, and this will become clear in the proof. 
\begin{theorem}
Let $u_0 \in {H}^p_{\text{per}}(\Omega)$ with $p > \frac{n}{2}$ and $|u_0|_{{H}^p_{\text{per}}(\Omega)} \le M_0$, for some $M_0 > 0$. There exists a continuous, positive function $T^*(\theta)$ defined for $|\theta| < \frac{\pi}{2}$ that depends only on the parameters $M_0$, $\nu$, and $L$, such that equation \eqref{20} with initial condition \eqref{21} has a unique regular solution $u$ inside the domain:
\[
\mathcal{D}_{T^*} := \left\{\zeta = s e^{i\theta} \in \mathbb{C} :\ 0 < s < T^*(\theta), \ \ |\theta| < \frac{\pi}{2}\right\},
\]
and the map
\[
\zeta \longmapsto e^{\sigma A^{1/2}} A^{p/2} u(\zeta), \quad \text{with } \sigma = \sigma(s,\theta) = s\cos\theta,
\]
is analytic in $\mathcal{D}_{T^*}$.
\end{theorem}
\begin{proof}
The proof follows the same strategy as that of Theorem 1. Therefore, we present only the essential details.

As in inequality \eqref{11}, choosing $\epsilon=\dfrac{\nu\cos\theta}{4}$ and $C_{\epsilon}=\dfrac{1}{\nu\cos\theta}$, we get
\begin{align}
& \frac{1}{2} \frac{\dd}{\dd s} |u^N|_{G_\sigma^{p/2}}^2 + \frac{3\nu \cos\theta}{4}|u^N|_{G_\sigma^{\frac{p+1}{2}}}^2 -\frac{1}{\nu \cos\theta}|u^N|_{
G_\sigma^{p/2}
}^2 \notag \\
&\leq -\operatorname{Re}\left( e^{i\theta}
\left( A^{p/2} e^{\sigma A^{1/2}} P_N F(u^N,\nabla u^N), A^{p/2} e^{\sigma A^{1/2}} u^N\right)_ {L^2(\Omega)}\right).\label{24}
\end{align} 
Using Lemma 1, an estimate similar to that derived in Lemma 2, \eqref{22} and \eqref{23}, the absolute value of the right-hand side of \eqref{24}is
\begin{align*}
&\Big| \operatorname{Re}\Big( e^{i\theta} \big( A^{p/2} e^{\sigma A^{1/2}} P_N F(u^N,\nabla u^N), A^{p/2} e^{\sigma A^{1/2}} u^N \big)_{L^2(\Omega)} \Big) \Big|
\notag \\
&\quad\leq  |P_NF(u^N,\nabla u^N)|_{\Gspace}|u^N|_{\Gspace}\notag \\
&\quad \leq  |F(u^N,\nabla u^N)|_{\Gspace}|u^N|_{\Gspace}\notag \\ &\quad  \leq C g(C_1 |u^N|_{\Gspace},C_2 |u^N_{x_1}|_{\Gspace},...,C_n |u^N_{x_{n+1}}|_{\Gspace})|u^N|_{\Gspace}
\notag \\ &\quad \leq  Cf(C_1 |u^N|_{\Gspace})|u^N|_{\Gspace}|\nabla u|^\gamma_{\Gspace}+ C  f(C_1 |u^N|_{\Gspace})|u^N|_{\Gspace}  
\end{align*}
Since $|\nabla u|_{G_{\sigma}^{\frac{p}{2}}}=| u|_{G_{\sigma}^{\frac{p+1}{2}}}$, $\gamma\in[0,2)$, using Young's inequality (see \cite{RobinsonRodrigoSadowski2016}) with  \[p=\dfrac{2}{\gamma},\quad q=\dfrac{2}{2-\gamma},\quad \epsilon=\dfrac{\nu \cos\theta}{4}, \quad C(\epsilon)=\dfrac{C}{(\nu \cos\theta)^{\frac{\gamma}{2-\gamma}}},\] the right-hand side of the above estimate is bounded by

\begin{align}
 & Cf(C_1 |u^N|_{\Gspace})|u^N|_{\Gspace}|u^N|^\gamma_{G_\sigma^{\frac{p+1}{2}}}+ C f(C_1 |u^N|_{\Gspace})|u^N|_{\Gspace}\notag\\ &\quad \leq \frac{C}{(\nu \cos\theta)^{\frac{\gamma}{2-\gamma}}} f(C_1 |u^N|_{\Gspace})^{\frac{2}{2-\gamma}}|u^N|_{\Gspace}^{\frac{2}{2-\gamma}}+\frac{\nu \cos\theta}{4} |u^N|^2_{G_\sigma^{\frac{p+1}{2}}}+C 
f(C_1 |u^N|_{\Gspace})|u^N|_{\Gspace}.\label{25}
\end{align}
Substituting \eqref{25} into \eqref{24}, we get
\begin{align*}
\frac{1}{2} \frac{\dd}{\dd s} |u^N|_{G_\sigma^{p/2}}^2 + \frac{\nu \cos\theta}{2}|u^N|_{G_\sigma^{\frac{p+1}{2}}}^2 \notag 
 &\leq \frac{C}{(\nu \cos\theta)^{\frac{\gamma}{2-\gamma}}} f(C_1 |u^N|_{\Gspace})^{\frac{2}{2-\gamma}}|u^N|_{\Gspace}^{\frac{2}{2-\gamma}} \\&+C 
f(C_1 |u^N|_{\Gspace})|u^N|_{\Gspace}+\frac{1}{\nu \cos\theta}|u^N|_{
G_\sigma^{p/2}
}^2
\end{align*}
We now follow the same steps in the proof of Theorem 1, starting from Inequality \eqref{12}, to finish the proof of Theorem 2. 
\end{proof}
\subsubsection*{2.2 The case $ p > \frac{n}{2}+1.$}
In the proof of the earlier case, there was only one way to bound the term $|\nabla u^N|^{\gamma}_{\Gspace}$ through the quantity on the left hand side $|u^N|^2_{G_\sigma^{\frac{p+1}{2}}}$, where the condition $\gamma \in [0,2)$ was essential to apply Young's inequality. However, if we assume smoother initial data $u_0\in H^p_{per}$ with $ p > \frac{n}{2}+1$, then no restriction on the structure of nonlinearity is needed for Theorem 2 to remain valid, since  Lemma 1 implies that the linear space $G_{\sigma}^{\frac{p-1}{2}}(\Omega)$ is a Banach algebra for every 
$\sigma > 0$ and \( p > \frac{n}{2}+1 \). We now consider equation \eqref{20} with general analytic nonlinearity and majorising function $g$ as in \eqref{22}.
\begin{theorem}
Let $u_0 \in {H}^p_{\text{per}}(\Omega)$ with \( p > \frac{n}{2}+1 \) and $|u_0|_{{H}^p_{\text{per}}(\Omega)} \le M_0$, for some $M_0 > 0$. There exists a continuous, positive function $T^*(\theta)$ defined for $|\theta| < \frac{\pi}{2}$ that depends only on the parameters $M_0$, $\nu$, and $L$, such that equation \eqref{20} with initial condition \eqref{21} has a unique regular solution $u$ inside the domain:
\[
\mathcal{D}_{T^*} := \left\{\zeta = s e^{i\theta} \in \mathbb{C} :\ 0 < s < T^*(\theta), \ \ |\theta| < \frac{\pi}{2}\right\},
\]
and the map
\[
\zeta \longmapsto e^{\sigma A^{1/2}} A^{p/2} u(\zeta), \quad \text{with } \sigma = \sigma(s,\theta) = s\cos\theta,
\]
is analytic in $\mathcal{D}_{T^*}$.
\end{theorem}

\begin{proof}
As the proof strategy mirrors that of Theorems 1 and 2, we present only the key details.
Following Inequality \eqref{11} we have
\begin{align}
& \frac{1}{2} \frac{\dd}{\dd s} |u^N|_{G_\sigma^{p/2}}^2 + (\nu \cos\theta -\epsilon)|u^N|_{G_\sigma^{\frac{p+1}{2}}}^2 -C_{\epsilon}|u^N|_{
G_\sigma^{p/2}
}^2 \notag \\
&\leq -\operatorname{Re}\left( e^{i\theta}
\left( A^\frac{p-1}{2} e^{\sigma A^{1/2}} P_N F(u^N,\nabla u^N), A^\frac{p+1}{2} e^{\sigma A^{1/2}} u^N \right)_ {L^2(\Omega)}\right),\label{26}
\end{align}
where $\epsilon$ to be chosen later.

Using Lemma 1, Lemma 2, Young's inequality, the fact that  $|\nabla u|^2_{G_{\sigma}^{\frac{p-1}{2}}}\leq | u|^2_{G_{\sigma}^{\frac{p}{2}}}$  and the majorizing function $g$ as in \eqref{22}, the absolute value of the  right-hand side of \eqref{26}is
\begin{align}
&\Big|-\operatorname{Re}\left( e^{i\theta}
\left( A^{\frac{p-1}{2}} e^{\sigma A^{1/2}} P_N F(u^N,\nabla u^N), A^{\frac{p+1}{2}} e^{\sigma A^{1/2}} u^N \right)_ {L^2(\Omega)}\right)\Big| \notag \\
&\quad\leq  |P_NF(u^N,\nabla u^N)|_{G^{\frac{p-1}{2}}_{\sigma}}|u^N|_{G^{\frac{p+1}{2}}_{\sigma}}\notag \\
&\quad \leq  |F(u^N,\nabla u^N)|_{G^{\frac{p-1}{2}}_{\sigma}}|u^N|_{G^{\frac{p+1}{2}}_{\sigma}}\notag \\
&\quad \leq C_{\epsilon}|F(u^N,\nabla u^N)|^2_{G^{\frac{p-1}{2}}_{\sigma}}+\epsilon|u^N|^2_{G^{\frac{p+1}{2}}_{\sigma}}\notag \\
&\quad \leq C_{\epsilon}C g^2(C_1 |u^N|_{G_\sigma^{\frac{p-1}{2}}},C_2 |u^N_{x_1}|_{G_\sigma^{\frac{p-1}{2}}},...,C_{n+1} |u^N_{x_n}|_{G_\sigma^{\frac{p-1}{2}}})+\epsilon|u^N|^2_{G^{\frac{p+1}{2}}_{\sigma}}\notag \\
&\quad \leq C_{\epsilon}C g^2(C_1 |u^N|_{G_\sigma^{\frac{p}{2}}},C_2 |u^N|_{G_\sigma^{\frac{p}{2}}},...,C_{n+1} |u^N|_{G_\sigma^{\frac{p}{2}}})+\epsilon|u^N|^2_{G^{\frac{p+1}{2}}_{\sigma}}.\label{27}
\end{align}
Choosing $\epsilon=\dfrac{\nu \cos\theta}{4},C_{\epsilon}=\dfrac{1}{\nu \cos\theta}$, from \eqref{26} and \eqref{27}, we obtain:
\begin{align*}
\frac{1}{2} \frac{\dd}{\dd s} |u^N|^2_{\Gspace} + \frac{\nu\cos\theta}{2}|u^N|_{G_\sigma^{\frac{p+1}{2}}}^2 &\leq \frac{C}{\nu \cos\theta} g^2(C_1 |u^N|_{G_\sigma^{\frac{p}{2}}},C_2 |u^N|_{G_\sigma^{\frac{p}{2}}},...,C_{n+1} |u^N|_{G_\sigma^{\frac{p}{2}}})\notag\\&\quad+\frac{1}{\nu \cos\theta}|u^N|_{
G_\sigma^{p/2}
}^2.
\end{align*}

Now, we proceed as in the proof of Theorem 1, starting from \eqref{12} to finish the proof of Theorem 3.
\end{proof}
\noindent \textbf{Alternative proof of Theorem 3}

We present here an alternative, simple proof of Theorem 3 by reducing the scalar equation to a vector system that satisfies the hypotheses of Theorem 2 as follows.

\begin{proof}
Assume $u(0)=u_0 \in H^p_{\text{per}}(\Omega)$ with $p > \frac{n}{2} + 1$. Define the vector field \[w = \nabla u \in \mathbb{C}^n.\] Differentiate the initial condition yields \[w(0) = \nabla u_0 = w_0.\] Since $u_0 \in H^p_{\text{per}}(\Omega)$ with $p > \frac{n}{2} + 1$, it follows that $w_0 \in H^{p-1}_{\text{per}}(\Omega)$ with $p-1 > \frac{n}{2}$. By the continuous embedding of Sobolev spaces on compact domains, $H^p_{\text{per}}(\Omega) \subset H^{p-1}_{\text{per}}(\Omega)$, which implies that both $u_0, w_0 \in H^{p-1}_{\text{per}}(\Omega)$ with $p-1 > \frac{n}{2} $.

Take the spatial gradient of the scalar equation \eqref{20}, we get the following system:
\begin{align}
    \frac{du}{d\zeta} + \nu Au + F(u,w) &= 0,\quad u(0)=u_0, \label{28} \\
    \frac{dw}{d\zeta} + \nu Aw + \nabla ( F(u,w)) &= 0,\quad w(0)=\nabla u_0=w_0
    .\label{29}
\end{align}
Apply the  chain rule to the function $\nabla F(u, w)$, we obtain 
\begin{align}
    \nabla ( F(u,w)) = \frac{\partial F(u,w)}{\partial u} \nabla u + \sum_{j=1}^n \frac{\partial F(u,w)}{\partial w_j} \nabla w_j.\label{30}
\end{align}
Consider the vector variable $W = (u, w)^T \in \mathbb{C}^{1+n}$. Since $\nabla u = w$, then \eqref{28}-\eqref{29} is a standard $(1+n) \times (1+n)$ closed nonlinear system of the form:
\begin{align}
    \frac{dW}{d\zeta} + \nu \mathbf{A}W + \mathcal{F}(W, \nabla W) = 0, \quad W(0) = W_0 = \begin{pmatrix} u_0 \\ w_0 \end{pmatrix},\label{31}
\end{align}
where $\mathbf{A} = \operatorname{diag}(A, A, \dots, A)$, and the total vector nonlinearity $\mathcal{F} = (\mathcal{F}_0, \mathcal{F}_1, \dots, \mathcal{F}_n)^T$ is defined component-wise by:
\begin{align}
    \mathcal{F}_0(W, \nabla W) &= F(u,w), \label{32}\\
    \mathcal{F}_i(W, \nabla W) &= \frac{\partial F(u,w)}{\partial u} w_i + \sum_{j=1}^n \frac{\partial F(u,w)}{\partial w_j} \frac{\partial w_j}{\partial x_i} \quad (\text{for } i=1,\dots,n).\label{33}
\end{align}
Since $F(u,w)$ is a holomorphic entire function in all its arguments, its partial derivatives $\dfrac{\partial F}{\partial u}$ and $\dfrac{\partial F}{\partial w_j}$ are holomorphic entire functions. Consequently, the mapping $\mathcal{F}: \mathbb{C}^{1+n} \times \mathbb{C}^{(1+n) \times n} \to \mathbb{C}^{1+n}$ is holomorphic entire function.

Because $W_0 \in H^{p-1}_{\text{per}}(\Omega)$ with $\frac{n}{2} < p-1 \leq \frac{n}{2}+1$, this vector framework maps exactly onto the structural requirements of Theorem 2. The local analyticity of $W(\zeta)$, and consequently of $u(\zeta)$, over the domain $\mathcal{D}_{T^*}$ follows immediately.
\end{proof}

\subsection*{4. Conclusion}

We establish analyticity in time for analytic nonlinear parabolic equations, combining ideas from \cite{FerrariTiti1998}, \cite{FoiasTemam1979} and \cite{FoiasTemam1989}. We treat two cases of nonlinearity, $F(u)$ and $F(u,\nabla u)$. In the case of $F(u,\nabla u)$, we give two different proofs. Remarkably, combining the approach developed here with \cite{Cao-Rammaha1} and \cite{Cao-Rammaha2} one can achieve similar analyticity results for the case when system \eqref{System} is considered on the two-dimensional sphere $\mathbb{S}^2$ instead of periodic boundary conditions. Furthermore, we recall that a different more elaborate approach for establishing similar analyticity results for semi-linear analytic Parabolic systems  in the whole space $\mathbb{R}^n$ was reported in \cite{Takac}.

Given the generality of our entire analytic formulation, this framework can be directly applied to a wide array of physically significant nonlinear parabolic systems, including: Generalized Burgers-type Equations \cite{Cole1951}, Kuramoto--Sivashinsky \cite{Kuramoto1978}, Cahn--Hilliard Variants \cite{Wu2022CahnHilliard}, and Nonlinear Reaction--Diffusion Systems \cite{Turing1952}.

\end{document}